\documentclass[11pt,reqno]{amsart}

\usepackage{amssymb}
\usepackage{latexsym}
\usepackage{amsbsy}
\usepackage{amsfonts}
\usepackage{appendix}
\usepackage{pstricks}
\usepackage{pstcol,pst-fill,pst-grad}
\usepackage{enumitem}

\numberwithin{equation}{section}
\newtheorem{theorem}{Theorem}[section]
\newtheorem{corollary}[theorem]{Corollary}
\newtheorem{lemma}[theorem]{Lemma}
\newtheorem{proposition}[theorem]{Proposition}

\theoremstyle{definition}

\theoremstyle{remark}

\newtheorem{remark}[theorem]{Remark}

\newtheorem{example}[theorem]{Example}

\newdimen\AAdi%
\newbox\AAbo%
\def\AAk#1#2{\setbox\AAbo=\hbox{#2}\AAdi=\wd\AAbo\kern#1\AAdi{}}%

\makeatletter
\def\eqlabel#1{\def\@currentlabel{#1}}

\def\formula#1{\def\@tempa{#1}\let\@tempb\theequation\def\theequation{%
\hbox{#1}}\def\@currentlabel{(\theequation)}$$}
\def\endformula{\leqno\hbox{(\@tempa)}$$\@ignoretrue\let\theequation\@tempb}

\def\given{\hskip5\p@\relax\vrule\@width.4\p@\hskip5\p@\relax}

\newcommand{\open}[1]{%
\par\normalfont\topsep6\p@\@plus6\p@\trivlist\item[\hskip\labelsep\itshape#1%
\@addpunct{.}]\ignorespaces}

\DeclareRobustCommand{\close}[1]{%
  \ifmmode 
  \else \leavevmode\unskip\penalty9999 \hbox{}\nobreak\hfill
  \fi
  \quad\hbox{$#1$}}
\makeatother

\newlength{\toskip}

\def\<{\langle}
\def\>{\rangle}

\def \Var {\textrm{Var}}
\def \Ent {\textrm{Ent}}
\def \Osc {\textrm{Osc}}

\begin{document}
\date{\today}

\title[Radial Inequalities.]{Poincar\'e and  Logarithmic Sobolev inequalities for nearly radial measures.}

 \author[P. Cattiaux]{\textbf{\quad {Patrick} Cattiaux $^{\spadesuit}$ \, \, }}
\address{{\bf {Patrick} CATTIAUX}\\ Institut de Math\'ematiques de Toulouse. CNRS UMR 5219. \\
Universit\'e Paul Sabatier,
\\ 118 route
de Narbonne, F-31062 Toulouse cedex 09.} \email{cattiaux@math.univ-toulouse.fr}

\author[A. Guillin]{\textbf{\quad {Arnaud} Guillin $^{\diamondsuit}$}}
\address{{\bf {Arnaud} GUILLIN}\\ Laboratoire de Math\'ematiques Blaise Pascal, CNRS UMR 6620, Universit\'e Clermont-Auvergne,
avenue des Landais, F-63177 Aubi\`ere.} \email{arnaud.guillin@uca.fr}

 \author[L. Wu]{\textbf{\quad {Liming} Wu $^{\diamondsuit}$}}
\address{{\bf {Liming} WU}\\ Laboratoire de Math\'ematiques Blaise Pascal, CNRS UMR 6620, Universit\'e Clermont-Auvergne,
avenue des Landais, F-63177 Aubi\`ere.} \email{liming.wu@uca.fr}

\maketitle

 \begin{center}

 \textsc{$^{\spadesuit}$  Universit\'e de Toulouse}
\smallskip

\textsc{$^{\diamondsuit}$ Universit\'e Clermont-Auvergne}
\smallskip

\end{center}

\begin{abstract}
If Poincar\'e inequality has been studied by Bobkov for radial measures, few is known about the logarithmic Sobolev inequalty in the radial case. We try to fill this gap here using different methods: Bobkov's argument and super-Poincar\'e inequalities, direct approach via $L_1$-logarithmic Sobolev inequalities. We also give various examples where the obtained bounds are quite sharp. Recent bounds obtained by Lee-Vempala in the logconcave  bounded case are refined for radial measures.

\end{abstract}
\bigskip

\bigskip

\bigskip

\section{Introduction}

Let $\mu(dx) = Z^{-1} \, e^{-V(x)} \, dx$ be a probability measure defined on $\mathbb R^n$ ($n\geq 2$). We do not require regularity for $V$ and allow it to take values in $\mathbb R \cup \{-\infty,+\infty\}$. We only require that $\int e^{-V} dx =1$, or more generally that the previous integral is finite. We denote by $\mu(f)$ the integral of $f$ w.r.t. $\mu$. 
\medskip

We will be interested in this note by functional inequalities verified by the measure $\mu$. Recall that $\mu$ satisfies a Poincar\'e inequality if for all smooth $f$,
\begin{equation}\label{eqpoinc}
\Var_\mu(f):= \mu(f^2) - \mu^2(f) \leq C_P(\mu) \, \mu(|\nabla f|^2) \, ,
\end{equation}
and that it satisfies a log-Sobolev inequality  if for all smooth $f$
\begin{equation}\label{eqLS}
\Ent_\mu(f^2):= \mu(f^2 \, \ln (f^2)) - \mu(f^2) \, \ln(\mu(f^2)) \leq C_{LS}(\mu) \, \mu(|\nabla f|^2) \, .
\end{equation}
$C_P$ and $C_{LS}$ are understood as the best constants for the previous inequalities to hold. We refer to \cite{Ane,BaGLbook,Wbook} among many others, for a comprehensive introduction to some of the useful consequences of these inequalities and their most important properties, such as convergence to equilibrium (in $L^2$ or in entropy) or concentration of measure.

If $\mu$ is not normalized as a probability measure, \eqref{eqpoinc} reads as 
\begin{equation}\label{eqpoincnonnormal}
\mu(f^2) \, - \, (1/\mu(\mathbb R^n)) \; \mu^2(f) \, \leq \, C_P(\mu) \, \mu(|\nabla f|^2) \, .
\end{equation}

One key feature of these inequalities is their tensorization property namely $$C_P(\mu\otimes\nu) = \max(C_P(\mu),C_P(\nu))$$ (the same for $C_{LS}$) giving a natural way to control these constants for product measures. 
\medskip

Another particular family of measures is the set of radial (or spherically symmetric) measures or more generally measures admitting a decomposition 
\begin{equation}\label{eqnearly}
\mu(dx) = \mu_r(d\rho) \, \mu_a(d\theta)
\end{equation}
with $x=\rho \, \theta$, $\rho \in \mathbb R^+$ and $\theta \in \mathbb S^{n-1}$. This amounts to $V(x)=V_r(\rho)+V_a(\theta)$ and $$\mu(dx) \, = \, n \, \omega_n \, \rho^{n-1} \, e^{-V_r(\rho)} \; e^{-V_a(\theta)} \, \sigma_n(d\theta)$$ where $\sigma_n$ denotes the uniform distribution on $\mathbb S^{n-1}$ and $\omega_n$ denotes the volume of the unit euclidean ball.

We shall call these measures \textit{nearly radial}. When $\mu_a=\sigma_n$ we simply say \textit{radial} and when $V_a$ is bounded below and above we will say \textit{almost radial}.
\medskip

It is natural to ask how to control $C_P(\mu)$ and $C_{LS}(\mu)$ in terms of constants related to $\mu_r$ and $\mu_a$. Since $\mu_a$ is supported by the sphere we will use the natural riemanian gradient, in other words, for $\theta \in \mathbb S^{n-1}$, we will decompose $$\nabla f = \nabla_\theta f + \nabla_{\theta^\perp} f:= \langle \nabla f,\theta\rangle + \Pi_{\theta^\perp} \nabla f$$ where $\Pi_{\theta^\perp}$ denotes the orthogonal projection onto $\theta^\perp$.
\medskip

Though natural it seems that the previous question was not often addressed in the literature with the notable exception of radial log-concave measures for which S. Bobkov (see \cite{bobsphere}) studied the Poincar\'e constant (his result is improved in \cite{BJM}) and for which Huet (see \cite{nolwen}) studied isoperimetric properties. 
\medskip

Our main results in the radial (or almost radial) case say that both the Poincar\'e and the log-Sobolev constant are controlled up to universal constants, by the corresponding constants for the radial part $\mu_r$ and $\mu(\rho^2)/(n-1)$ for Poincar\'e and some slightly more intricate combinaition for log-Sobolev, i.e.

\begin{theorem}
Let $\mu$ be a radial measure.
\begin{enumerate}
\item (Bobkov's result)
$$ C_P(\mu)\le \max\left(C_P(\mu_r),\frac{\mu(\rho^2)}{n-1}\right).$$
\item (Th. 4.5) there exists an universal constant $c$ such that
$$C_{LS}(\mu)\le c\left(C_{LS}(\mu_r)+\mu(\rho)\max\left(C_P(\mu_r),\frac{\mu(\rho^2)}{n-1}\right)^{1/2}\right).$$
\end{enumerate}
\end{theorem}

\smallskip

Other results and some consequences are also described. More precisely,
the ``tensorization'' part of Bobkov's proof is elementary and will be explained in Section \ref{secpoinc}, where we will also show how it applies to other types of Poincar\'e inequalities (weak or super). As a byproduct we will obtain a first (bad) bound for the log-Sobolev constant. 
\smallskip

In section \ref{secls} we propose a direct approach of the logarithmic Sobolev inequality for (almost) radial measures. This approach uses in particular $\mathbb L^q$ ($1 \leq q \leq 2$) log-Sobolev inequalities for the uniform measure on the sphere we establish in Section \ref{secsphere}, based on the study made in \cite{bobhoud} and results in \cite{emillogsob}. In the framework of general log-concave measures, similar ideas already appear in \cite{bob99}. 
\smallskip

All these results are applied in section \ref{secapp} to some examples. In particular, in the radial case, we improve upon the bound recently obtained by Lee and Vempala (\cite{leevemplogsob}) for compactly supported (isotropic) log-concave measures. It reads
\begin{theorem}[Th 5.9]
For any radial logconcave probability measure $\mu$ whose support $K$ is bounded then
$$C_{LS}(\mu)\le C \; \frac{diam^2(K)}{n-1}$$
for some universal constant $C$.
\end{theorem}
\medskip

\textbf{Additional notations.}
\smallskip

Let us recall the $\mathbb L^1$ inequalities we are interested in, namely Cheeger (or $\mathbb L^1$ Poincar\'e) inequality
\begin{equation}\label{eqdefcheeger}
\mu(|f-m_\mu f|) \leq C \, \int \, |\nabla f| \, d\mu
\end{equation}
where $m_\mu f$ denotes a $\mu$-median of $f$ and similarly the $\mathbb L^q$ log-Sobolev inequality ($1\leq q \leq 2$)
\begin{equation}\label{eqdefl1ls}
\mu(|f|^q \, \ln(|f|^q)) \, - \, \mu(|f|^q) \, \ln(\mu(|f|^q)) \leq \, C \, \int \, |\nabla f|^q \, d\mu \, .
\end{equation}
As usual we denote by $C_C(\mu)$ and $C_{LSq}(\mu)$ the optimal constants in the previous inequalities. It is known (see e.g. \cite{bob99}) that $C_P(\mu) \leq 4 C^2_C(\mu)$. One can also show that there exists an universal constant $D$ such that $C_{LS}(\mu) \leq D \, C^2_{LS1}(\mu)$ (see below).

These inequalities are strongly related to the isoperimetric profile of $\mu$. Recall that the isoperimetric profile $I_{\mu}$ of $\mu$ is defined for $p\in [0,1]$ as $$I_{\mu}(p) = \inf_{A \, \textrm{ s.t. } \mu(A)=p} \mu_n^+(\partial A)$$ where $$\mu_n^+(\partial A) = \liminf_{h \to 0} \, \frac{\mu(A^h) - \mu(A)}{h}$$ $A^h$ being the geodesic enlargement of $A$ of size $h$. Of course in ``smooth'' situations, as absolutely continuous measures w.r.t. the natural riemanian measure on a riemanian manifold, $I_\mu(p)=I_\mu(1-p)$ so that it is enough to consider $p \in[0,\frac 12]$.

The following results are then well known (see e.g. \cite{bob99} for the first one and \cite{bobhoud} Theorem 1.1 for the second one)
\begin{proposition}\label{propcheeg}
There is an equivalence between the following two statements:
\begin{enumerate} 
\item[1)] 
$$I_\mu(p) \geq C \, \min(p,1-p)$$ 
\item [2)] \quad and \eqref{eqdefcheeger} holds with constant $1/C$.
\end{enumerate}
There is an equivalence between the following two statements: 
\begin{enumerate}
 \item [3)] \quad $$ \textrm{for $p\in [0,1/2]$, } \quad I_\mu(p) \geq C \, p \, \ln(1/p)$$
 \item[4)] \quad and \eqref{eqdefl1ls} holds for $q=1$ and with constant $1/C$.
\end{enumerate}
\end{proposition}
\smallskip

According to the previous proposition a $\mathbb L^1$ log-Sobolev inequality implies $$I_\mu(p) \geq \, (1/C_{LS1}(\mu)) \, p \, \ln(1/p) \geq \ln(2) \, (1/C_{LS1}(\mu))  \, p \, \ln^{1/2}(1/p)$$ for $p\in [0,1/2]$. According to the results in \cite{bobzeg} the latter implies that $C_{LS}(\mu) \leq  \, C_{LS1}^2(\mu)$ for some new universal constant $D$.
\medskip

\section{The uniform measure on the sphere.}\label{secsphere}

In this section we shall recall some properties of $\sigma_n$ the uniform measure on the unit sphere $\mathbb S^{n-1}$. In what follows $s_n$ denotes the area of $\mathbb S^{n-1}$ which is equal to $\frac{2 \pi^{n/2}}{\Gamma(n/2)}$.
\medskip

Many properties rely on the fact that the sphere $\mathbb S^{n-1}$ satisfies the curvature-dimension condition $CD(n-2,n-1)$ (see \cite{BaGLbook} p.87). It follows from Proposition 4.8.4 and Theorem 5.7.4 in \cite{BaGLbook} that for $n \geq 3$, $$C_P(\sigma_n) \leq \, \frac{1}{n-1} \quad \textrm{ and } \quad C_{LS}(\sigma_n) \leq \, \frac{2}{n-1} \, .$$ These bounds are also true for $n=2$. It is easy to check for the Poincar\'e constant using e.g. \cite{BaGLbook} Proposition 4.5.5 iii). For the logarithmic Sobolev inequality see \cite{EY}. Actually these bounds are optimal (at least up to some universal constants).
\medskip

$I_{\sigma_n}$ is described in by Bobkov and Houdr\'e in \cite{bobhoud} Lemma 9.1 and Lemma 9.2
\begin{proposition}
Let $n\geq 3$. Define $$f_n(t) = \frac{s_{n-1}}{s_n} \, (1-t^2)^{\frac{n-3}{2}} \, ; \; -1\leq t\leq 1 \, .$$ Let $F_n$ be the distribution function on $[-1,1]$ whose probability density is $f_n$, $G_n=F^{-1}_n$ be the inverse function of $F_n$. Then $$I_{\sigma_n}(p)=\frac{s_{n-1}}{s_n} \, (1-G_n^2(p))^{\frac{n-2}{2}} \, .$$
\end{proposition}
Notice in particular that $$I_{\sigma_n}(1/2)= \frac{s_{n-1}}{s_n}=\frac{1}{\sqrt \pi} \; \frac{\Gamma(n/2)}{\Gamma((n-1)/2)} \, .$$ Using the extension of Stirling'formula to the $\Gamma$ function, one sees that 
\begin{equation}\label{eqequiv}
\lim_{n \to +\infty} \;  \sqrt{2 \, \pi/(n-1)} \; \; \frac{s_{n-1}}{s_n} \, = \, 1 \; ,
\end{equation}
so that one can find universal constants $c$ and $C$ such that $$c \, \sqrt {n-1} \, \leq \frac{s_{n-1}}{s_n}
\, \leq \, C \, \sqrt {n-1} \, .$$

Using their Lemma 8.2, Bobkov and Houdr\'e (\cite{bobhoud}) also show in their section 9 that the minimum of $I_{\sigma_n}(p)/p(1-p)$ is attained for $p=1/2$ so that for all $p\in [0,1/2]$
\begin{equation}\label{eqcheeger1sphere}
I_{\sigma_n}(p) \, \geq \, \frac 12 \, \frac{s_{n-1}}{s_n} \, p
\end{equation}
and 
\begin{equation}\label{eqcheeger2sphere}
C_C(\sigma_n) \, = \, 2 \, \frac{s_{n}}{s_{n-1}} \, \leq \, \frac{C}{\sqrt{n-1}}
\end{equation}
for some universal constant $C$. Actually, the application of Lemma 8.2 in \cite{bobhoud} to $I_{\sigma_n}^\alpha$ for $\alpha \in [1,n/(n-1)]$ furnishes some Sobolev inequality (see Proposition 8.1 in \cite{bobhoud}).

\begin{remark}\label{remcheeg1}
Since the curvature dimension condition $CD(n-2,n-1)$ implies $CD(0,+\infty)$ for $n \geq 2$ (i.e. $\sigma_n$ is log-concave) it is known (see e.g. \cite{emil1} for numerous references) that $I_{\sigma_n}$ is concave on $[0,1/2]$. This furnishes another proof of \eqref{eqcheeger1sphere}.

Actually for any log-concave measure $\mu$ it was shown by Ledoux (\cite{ledgap} formula (5.8)) that $C_C(\mu) \leq 6 \, C_P^{\frac 12}(\mu)$. The constant $6$ is improved by $16/\sqrt{\pi}$ in \cite{CGlogconc} proposition 2.11. We thus have the precise estimate $C_C(\sigma_n) \leq (16/\sqrt \pi) \, (1/\sqrt{n-1})$ and a lower bound for the isoperimetric profile linked to the Poincar\'e constant. Notice that the asymptotic ($n \to +\infty$) optimal constant here is $2/\sqrt \pi$. \hfill $\diamondsuit$
\end{remark}
\medskip

Another remarkable property of log-concave measures identified by Ledoux (\cite{ledgap}  Theorem 5.3) for the usual ($\mathbb L^2$) log-Sobolev inequality and generalized in Theorem 1.2 of \cite{emillogsob} by E. Milman, is that the $\mathbb L^q$ log-Sobolev inequality also furnishes such a control for the isoperimetric profile, more precisely for any log-concave probability measure $\mu$ and $p\in [0,1/2]$, $$I_\mu(p) \geq \frac{\sqrt 2}{34 \, \sqrt{C_{LS}(\mu)}} \, p \, \ln^{\frac 12}(1/p) \, ,$$ in the case $q=2$, and more generally there exists an universal constant $c$ such that for all $1\leq q \leq 2$, $$I_\mu(p) \geq \frac{c}{C^{1/q}_{LSq}(\mu)} \, p \, \ln^{\frac 1q}(1/p) \, .$$

The converse statement $I_\mu(p) \geq c_\mu \,  \, p \, \ln^{\frac 1q}(1/p)$ implies $C_{LSq}(\mu) \leq C \, \frac{1}{c_\mu^q}$ for some universal constant $C$ does not require log-concavity and was shown by Bobkov-Zegarlinski \cite{bobzeg}. 

Our goal is now to determine the best possible constant $C_n(q)$ (best in terms of the dimension) such that 
\begin{equation}\label{eqlsqsphere1}
I_{\sigma_n}(p) \geq C_n(q) \, p \, \ln^{1/q}(1/p) \, ,
\end{equation}
 and then to apply the equivalence we explained before to derive the best possible $\mathbb L^q$ log-Sobolev inequality for $\sigma_n$. 
\medskip

We shall consider $p=F_n(x)$ in order to have a tractable expression $$I_{\sigma_n}(F_n(x))= \frac{s_{n-1}}{s_n} \, (1-x^2)^{\frac{n-2}{2}} \, .$$

To this end we will use the following elementary lemma
\begin{lemma}\label{lemliming}
For all $x \in [-1,0)$ define $A_n(x)= \frac{1}{(n-1)} \, (1-x^2)^{1/2} \, I_{\sigma_n}(F(x))$. It holds $$A_n(x ) \, \leq \, F_n(x) \, \leq \, \frac{A_n(x)}{-x} \, .$$ 
\end{lemma}
\begin{proof}
For all $x \in [-1,0)$, we have on one hand,
\begin{eqnarray*}
F_n(x) &=& \int_{-1}^x \, \frac{s_{n-1}}{s_n} \, (1-u^2)^{(n-3)/2} \, du \\ &\leq& \int_{-1}^x \, \frac{s_{n-1}}{s_n} \, \frac{-u}{-x} \, (1-u^2)^{(n-3)/2} \, du \\ &=& \frac{s_{n-1}}{(-x)(n-1)s_n} \, (1-x^2)^{(n-1)/2} \, = \, \frac{1}{(-x)(n-1)} \, (1-x^2)^{1/2} \, I_{\sigma_n}(F(x)) \, .
\end{eqnarray*}
On the other hand
\begin{eqnarray*}
F_n(x) &=& \int_{-1}^x \, \frac{s_{n-1}}{s_n} \, (1-u^2)^{(n-3)/2} \, du \\ &\geq& \int_{-1}^x \, \frac{s_{n-1}}{s_n} \, (-u) \, (1-u^2)^{(n-3)/2} \, du \\ &=& \frac{s_{n-1}}{(n-1)s_n} \, (1-x^2)^{(n-1)/2} \, = \, \frac{1}{(n-1)} \, (1-x^2)^{1/2} \, I_{\sigma_n}(F(x)) \, .
\end{eqnarray*}
\end{proof}
It follows for $F_n(x)\leq 1/2$ and provided $(-x)/A_n(x) > 1$ (for its logarithm to be positive), $$F_n(x) \, \ln^{1/q}(1/F_n(x)) \, \geq \, F_n(x) \, \ln^{1/q}((-x)/A_n(x))  \, \geq \, A_n(x) \, \ln^{1/q}((-x)/A_n(x)) \, . $$ But $$\ln((-x)/A_n(x)) \geq \ln(-x) + \frac{n-1}{2} \, \ln(1/(1-x^2)) \, \geq \, \ln(-x) +  \frac 14 \, \ln(1/(1-x^2)) \geq 0$$ for all $x \in (-1,-a]$ for some $0<a<1$ using continuity, so that for $x\leq -a$, $$\ln(-x) + \frac{n-1}{2} \, \ln(1/(1-x^2)) \geq \, \frac{2n-3}{4} \, \ln(1/(1-x^2)) \, .$$
This yields for such an $x$, $$F_n(x) \, \ln^{1/q}(1/F_n(x)) \, \geq \, \frac{1}{(n-1)} \, (1-x^2)^{1/2} \, I_{\sigma_n}(F(x)) \, \left(\frac{2n-3}{4} \, \ln(1/(1-x^2)) + \ln\left(\frac{(n-1)s_n}{s_{n-1}}\right)\right)^{1/q}$$ and finally that there exists a constant $c(q)$ such that for such an $x$,
\begin{equation}\label{eqnegat}
F_n(x) \, \ln^{1/q}(1/F_n(x)) \, \geq \, c(q) \, (n-1)^{\frac{q-1}{q}} \, I_{\sigma_n}(F_n(x)) \, .
\end{equation}
In particular the constant $C_n(q)$ in \eqref{eqlsqsphere1} cannot be bigger than a constant times $(n-1)^{\frac{q-1}{q}}$.
\medskip

For $q=1$, it is known (see Theorem 2 in \cite{ledlogconc}) that $$I_{\sigma_n}(p) \, \geq \, \frac{1}{2\pi} \, p \, \ln(1/p)$$ for $p \in[0,1/2]$, so that if not optimal, this result is optimal up to a constant.
\medskip

Since we cannot hope a better result, our goal will be now to prove that $$C_n(q) \geq C(q) \, (n-1)^{\frac{q-1}{q}} \, .$$ To this end we will take advantage of Ledoux's result applied to $\sigma_n$, i.e. 
\begin{equation}\label{eqleds}
I_{\sigma_n}(p) \, \geq C(2) \, \sqrt{n-1} \; p \, \ln^{1/2}(1/p) \, .
\end{equation}
Indeed, its is easy to check that $$\sqrt{n-1} \; p \, \ln^{1/2}(1/p) \, \geq \, (n-1)^{\frac{q-1}{q}} \, p \, \ln^{1/q}(1/p)$$ as soon as $p\geq e^{-(n-1)}$. It is thus enough to consider the remaining $p=F_n(x)\leq e^{-(n-1)}$. Using lemma \ref{lemliming} we thus have $$\frac{s_{n-1}}{(n-1)s_n} \, (1-x^2)^{(n-1)/2} \, \leq \, e^{-(n-1)}$$ so that 
\begin{eqnarray*}
\ln(1-x^2) \, &\leq& \, 2 \, \left(-1 + \frac{\ln((n-1)s_n/s_{n-1})}{n-1}\right) \\ &\leq& \, 2 \, \left(-1 + \frac{\ln(n-1)}{n-1}\right) \, \leq \, -2 \, \frac{e-1}{e} \, ,
\end{eqnarray*}
which implies
\begin{equation}\label{eqvalx}
x \, \leq \, - \, \left(1-e^{-2(e-1)/e}\right)^{\frac 12} \, = \, y \, .
\end{equation}
We can thus deduce, for such an $x$
\begin{eqnarray}\label{eqlsqsphere2}
F_n(x) \, \ln^{1/q}(1/F_n(x))  &\leq&  \frac{1}{(-x)(n-1)} \, (1-x^2)^{1/2} \, I_{\sigma_n}(F(x)) \, \nonumber \\ && \quad \left(\ln((n-1)s_n/s_{n-1}) + \frac{n-1}{2} \ln(1/(1-x^2))\right)^{1/q} \nonumber \\ &\leq& \frac{1}{(-y)(n-1)} \, (1-x^2)^{1/2} \, I_{\sigma_n}(F(x)) \, \nonumber \\ && \quad \left(\ln^{1/q}((n-1)s_n/s_{n-1}) + \left(\frac{n-1}{2}\right)^{1/q} \ln^{1/q}(1/(1-x^2))\right) \nonumber\\ &\leq& D \, (n-1)^{\frac{q-1}{q}} \, I_{\sigma_n}(F(x)) \, ,
\end{eqnarray}
for some constant $D$. We may thus state
\begin{proposition}\label{propls1sphere}
Let $n\geq 3$. For any $1\leq q \leq 2$ there exist constants $C$ and $C(q)$ such that for all $p \in [0,1/2]$ one has $$I_{\sigma_n}(p) \, \geq \, C \, (n-1)^{1-\frac 1q} \, p \, \ln^{1/q}(1/p) \, , $$ yielding $$ C_{LSq}(\sigma_n) \, \leq  \, C(q)  \, \left(\frac{1}{n-1}\right)^{q-1} \, .$$  Actually all constants can be chosen independently of $q\in [1,2]$. The bound is optimal with respect to the dimension.
\end{proposition}
\medskip

\section{Poincar\'e inequality and variants for nearly radial measures.}\label{secpoinc}

 Let us explain Bobkov's tensorization method.
\medskip

For $\mu(dx) = \mu_r(d\rho) \, \mu_a(d\theta)$ and a smooth $f$ one has first
\begin{eqnarray}\label{eqpoinrad1}
\int \, f^2(\rho \, \theta) \, \mu_r(d\rho) &\leq& C_P(\mu_r) \, \int \, \langle \nabla f(\rho \theta),\theta\rangle^2 \, \mu_r(d\rho) + \left(\int \, f(\rho \theta) \, \mu_r(d\rho)\right)^2 \, \nonumber \\ &=& C_P(\mu_r) \, \int \, |\nabla_\theta f(\rho \theta)|^2 \, \mu_r(d\rho) + \left(\int \, f(\rho \theta) \, \mu_r(d\rho)\right)^2 \, .
\end{eqnarray}

Integrating with respect to $\mu_a$ we obtain
\begin{equation}\label{eqpoinc1}
\mu(f^2) \, \leq \,  C_P(\mu_r)  \, \mu(|\nabla_\theta f|^2) + \int \, \left(\int \, f(\rho \theta) \, \mu_r(d\rho)\right)^2 \, \mu_a(d\theta) \, .
\end{equation}
But if we define $w(\theta)=\int \, f(\rho \theta) \, \mu_r(d\rho)$ it holds,
\begin{eqnarray}\label{eqpoinc2}
\int \, \left(\int \, f(\rho \theta) \, \mu_r(d\rho)\right)^2 \, \mu_a(d\theta) &\leq& C_P(\mu_a) \, \int |\nabla w(\theta)|^2 \, \mu_a(d\theta) + \left(\int \, w(\theta) \, \mu_a(d\theta)\right)^2 \nonumber \\ &\leq& C_P(\mu_a) \, \int \left|\int \, \rho \, \nabla_{\theta^\perp}f(\rho \theta) \mu_r(d\rho)\right|^2 \mu_a(d\theta) \, + \, \mu^2(f) \nonumber \\ &\leq& C_P(\mu_a) \, \mu_r(\rho^2) \, \mu(|\nabla_ {\theta^\perp} f|^2) \, + \, \mu^2(f) \, ,
\end{eqnarray}
where we have used the Cauchy-Schwarz inequality in the last inequality. Here we assume that the Poincar\'e constant of $\mu_a$ on the sphere $S^{n-1}$ is w.r.t. the usual gradient and not the gradient on the sphere. Using that $$|\nabla_\theta f|^2 + |\nabla_{\theta^\perp}f|^2=|\nabla f|^2 \, ,$$ we have thus obtained
\begin{theorem}\label{thmpoinc}
If $\mu(dx) = \mu_r(d\rho) \, \mu_a(d\theta)$ then $$C_P(\mu) \, \leq \, \max(C_P(\mu_r) \, ,\, \mu_r(\rho^2) \, C_P(\mu_a)) \, .$$
\end{theorem}
Recall that if $C_P(\mu_r)<+\infty$ then $\mu_r(e^{\lambda \rho})<+\infty$ for $\lambda<2/\sqrt{C_P(\mu_r)}$ (see e.g. \cite{Ane}) so that $\mu_r(\rho^2)$ is finite too.
\medskip

A weak version of the Poincar\'e inequality has been introduced in \cite{RW} (also see the related papers \cite{BCR2,bob07,CGGR}). A weak Poincar\'e inequality is a family of inequalities taking the form: for any $t>0$ and all smooth $f$, 
\begin{equation}\label{eqwp}
\Var_\mu(f) \, \leq \, \beta^{WP}_\mu(t) \, \mu(|\nabla f|^2) + \, t \, \Osc^2(f)
\end{equation}
where $\beta^{WP}_\mu$ is a non increasing function that can explode at $t=0$ (otherwise the classical Poincar\'e inequality is satisfied) and $\Osc(f)$ denotes the Oscillation of $f$. The previous proof shows that
\begin{theorem}\label{thmwpoinc}
If $\mu(dx) = \mu_r(d\rho) \, \mu_a(d\theta)$ then $$\beta^{WP}_\mu(t) \, \leq \, \max(\beta^{WP}_{\mu_r}(t/2) \, , \, \mu_r(\rho^2) \, \beta^{WP}_{\mu_a}(t/2)) \, .$$
\end{theorem}
The integrability of $\rho^2$ is ensured as soon as $\beta^{WP}_\mu$ does not explode too quickly at the origin (see e.g. \cite{BCR2}).
\medskip

Similarly one can reinforce the Poincar\'e inequality introducing super Poincar\'e inequalities (\cite{Wbook,Wang08,BCR1,BCR3,CGWW,BaGLbook}): for any $t\geq 1$ and all smooth $f$
\begin{equation}\label{eqsuperpgn}
\mu(f^2) \, \leq \, \delta_\mu(t) \, \mu(|\nabla f|^2) + \, t \, \mu^2(|f|) \, ,
\end{equation}
which is called a generalized Nash inequality in \cite{BaGLbook} Chapter 8.4. Here $\delta_\mu$ is assumed to be a non increasing function. It immediately follows $C_P(\mu)\leq \delta_\mu(1)$. If $\delta_\mu(t) \to 0$ as $t \to +\infty$ one may consider the inverse function $\beta^{SP}_\mu(t) = \delta^{-1}_{\mu}(t)$ defined for $t \in ]0,\delta_\mu(1)]$ and which is non increasing with values in $[1,+\infty[$. We can thus rewrite \eqref{eqsuperpgn} as: for $t \in ]0,\delta_\mu(1)]$,
\begin{equation}\label{eqsuperp}
\mu(f^2) \, \leq \, t \, \mu(|\nabla f|^2) + \, \beta^{SP}_\mu(t) \, \mu^2(|f|) \, .
\end{equation}
Conversely, assume that there exists a function $\beta$ (that can always be chosen non increasing) defined for $t>0$ and such that $$\mu(f^2) \, \leq \, t \, \mu(|\nabla f|^2) + \, \beta(t) \, \mu^2(|f|)$$ for all $t>0$ and nice function $f$. Applying this inequality to constant functions shows that $\beta(t)\geq 1$ for all $t$. But we may replace $\beta(t)$ by $1$ as soon as $t\geq C_P(\mu)$.
\smallskip

It is known that if $\beta^{SP}_\mu(t)$ behaves like $c e^{c'/t}$ as $t \to 0$, \eqref{eqsuperp} together with a Poincar\'e inequality is equivalent to the logarithmic Sobolev inequality (see below). 
\medskip

Following the same route we immediately get that for any positive $t$ and $s$,
\begin{equation}\label{eqsuper1}
\mu(f^2) \leq (s+\beta^{SP}_{\mu_r}(s) \, t \, \mu_r(\rho^2)) \, \mu(|\nabla f|^2) + \beta^{SP}_{\mu_r}(s) \, \beta^{SP}_{\mu_a}(t) \, \mu^2(|f|) \, .
\end{equation}
Of course this is a new super Poincar\'e inequality or more precisely a new family of super Poincar\'e inequalities. The most natural choice (not necessarily the best one) is $$t = \frac{s }{\mu(\rho^2) \, \beta_{\mu_r}^{SP}(s)}$$ yielding the following
\begin{theorem}\label{thmspoinc}
If $\mu(dx) = \mu_r(d\rho) \, \mu_a(d\theta)$ then $$\beta^{SP}_\mu(t) \, \leq \, \beta^{SP}_{\mu_r}(t/2)  \, \beta^{SP}_{\mu_a}\left(\frac{t}{ 2 \, \mu(\rho^2)\, \beta^{SP}_{\mu_r}(t/2)}\right) \, .$$
\end{theorem}
\bigskip

\subsection{The radial case. \\ \\}\label{secradial}

If $\mu$ is radial i.e. $\mu_a=\sigma_n$ the uniform measure on $\mathbb S^{n-1}$, we deduce from $C_P(\sigma_n)\leq \frac{1}{n-1}$ (see section \ref{secsphere}), that 
\begin{equation}\label{eqpoincrad}
C_P(\mu) \, \leq \, \max\left(C_P(\mu_r) \, , \, \frac{\mu_r(\rho^2)}{n-1}\right) \, .
\end{equation} 

Actually $\sigma_n$ satisfies the much stronger Sobolev inequality (for $n\geq 4$ see \cite{BaGLbook} p.308 written for spherical gradient but recall the introduction)
\begin{equation}\label{eqsigmasob}
\parallel g\parallel_{\frac{2n-2}{n-3}}^2 \, \leq \, \sigma_n(g^2) + \frac{4}{(n-1)(n-3)} \, \sigma_n(|\nabla g|^2) \, .
\end{equation}
We deduce from this and the Poincar\'e inequality, the following inequality
\begin{equation}\label{eqsigmasobbis}
\parallel g\parallel_{\frac{2n-2}{n-3}}^2 \, \leq \, \sigma^2_n(g) + \frac{c_n}{n-1} \, \sigma_n(|\nabla g|^2) \, ,
\end{equation}
with $c_n=\frac{n+1}{n-3}$.

\noindent One can thus derive the corresponding $\beta^{SP}_{\sigma_n}$. If one wants to see the dimension dependence one has to be a little bit careful. 

\noindent First we apply H\"{o}lder's inequality for $p>2$, $$\sigma_n(g^2) \; \leq \; \sigma_n^{\frac{1}{p-1}}(|g|^p) \; \sigma_n^{\frac{p-2}{p-1}}(|g|) \, ,$$ then choose $p=\frac{2n-2}{n-3}$, yielding according to what precedes 
\begin{eqnarray*}
\sigma_n(g^2) \; &\leq& \; \left(\sigma_n^2(g) + \frac{c_n}{n-1} \sigma_n(|\nabla g|^2)\right)^{\frac{n-1}{n+1}} \; \sigma_n^{\frac{4}{n+1}}(|g|) \, \\ &\leq& \; \sigma_n^2(|g|) \; + \; \left(\frac{c_n}{n-1} \sigma_n(|\nabla g|^2)\right)^{\frac{n-1}{n+1}} \; \sigma_n^{\frac{4}{n+1}}(|g|) \, .
\end{eqnarray*}
Recall Young's inequality: for all $p>1$ and $a,b,t>0$, $$ab \leq t \, \frac{a^p}{p} + t^{-(q-1)} \, \frac{b^q}{q}$$ with $1/p + 1/q = 1$. We deduce from what precedes, this time with $p=(n+1)/(n-1)$,
$$\sigma_n(g^2) \; \leq \; t \, \sigma_n(|\nabla g|^2) \, + \, \left(1 \, + \, \frac{1}{(n+3)^{(n-1)/2}} \, t^{-(n-1)/2}\right) \, \sigma_n^2(|g|)$$ and finally that 
\begin{equation}\label{betaspunif}
\beta^{SP}_{\sigma_n}(t) \, \leq \, 1 + \frac{1}{(n+3)^{(n-1)/2}} \; t^{-(n-1)/2} \, .
\end{equation} 

We can thus plug \eqref{betaspunif} in Theorem \ref{thmspoinc} and get 
\begin{corollary}\label{corspoinc}
If $\mu(dx) = \mu_r(d\rho) \, \sigma_n(d\theta)$ and $n\geq 4$, then $$\beta^{SP}_\mu(t) \, \leq \, \beta^{SP}_{\mu_r}(t/2)  \, \left(1 + \left(\frac{\mu(\rho^2) \, \beta^{SP}_{\mu_r}(t/2)}{(n+3)}\right)^{(n-1)/2} \right) \, .$$
\end{corollary}
The cases $n=2$ and $n=3$ can be studied separately.
\bigskip

\subsection{Application to the log-Sobolev inequality. \\ \\}\label{secsuperradial}

In this subsection we assume that $\mu$ is (almost) radial as in the previous subsection.
\medskip

The equivalence between a log-Sobolev inequality and a super-Poincar\'e inequality is well known (see e.g. \cite{Wbook} Theorems 3.3.1 and 3.3.3, despite some points we do not understand in the proofs). But here we need precise estimates on the constants. One way to get these estimates is to use the capacity-measure description of these inequalities following the ideas in \cite{BCR3} (also see \cite{Zitt} for some additional comments). 
\begin{lemma}\label{lemlssp}[see \cite{CGG} Proposition 3.4]\\
If $\mu$ satisfies a logarithmic-Sobolev inequality with constant $C_{LS}(\mu)$ then $\beta^{SP}_\mu(t) \leq \, 2  \, e^{2 \, C_{LS}(\mu)/t}$.
\end{lemma}
A simplified version of the results of \cite{BCR3,BCR1} is contained in \cite{BaGLbook}, see in particular Proposition 8.3.2 and Proposition 8.4.1, from which one can deduce the (slightly worse bound) $\beta^{SP}_\mu(t) \leq \, e^{96 \, C_{LS}(\mu)-2/t}$. 

\noindent The converse part is a consequence of \cite{BaGLbook}.
\begin{lemma}\label{lemspls}[see \cite{BaGLbook} Proposition 8.3.2 and Proposition 8.4.1] \\
Conversely, if $\beta^{SP}_\mu(t) \leq C_1 \, e^{C_2/t}$ and $\mu$ satisfies a Poincar\'e inequality with constant $C_P(\mu)$, then $$C_{LS}(\mu) \leq 64 \, \left(C_2 \, + \, \ln\left(1\vee \frac{(1+2e^2)C_1}{4}\right) \, C_P(\mu)\right) \, .$$
\end{lemma}
\begin{proof}
Recall that it is enough to look at $t\leq C_P(\mu)$ and then take $\beta_\mu^{SP}(t)=1$ for $t\geq C_P(\mu)$. So we may replace $C_1 \, e^{C_2/t}$ by the larger $ \frac{4}{1+2e^2} \, e^{C'_2/t}$ with $$C'_2=C_2+C_P(\mu) \, \ln\left(1\vee \frac{(1+2e^2)C_1}{4}\right) \, .$$
In \cite{BaGLbook} terminology we thus have $\delta(s)=C'_2/\ln(s(1+2e^2)/4)$, that satisfies the assumptions of Proposition 8.4.1 in \cite{BaGLbook} with $q=4$. We thus get for $\mu(A) \leq 2$, $$Cap_\mu(A) \geq \frac{\mu(A) \, \ln((1+2e^2)/2\mu(A))}{8 C'_2} \, .$$ But for $\mu(A) \leq 1/2$, $$\ln((1+2e^2)/2\mu(A)) \geq \ln\left(1 + \frac{e^2}{\mu(A)}\right) \, .$$ We may thus apply Proposition 8.3.2 in \cite{BaGLbook} yielding $C_{LS}(\mu) \leq 64 C'_2$.
\end{proof}
Hence if $\mu_r$ satisfies a log-Sobolev inequality, $\beta^{SP}_{\mu_r}(t)\leq 2 e^{2 C_{LS}(\mu_r)/t}$ so that using Corollary \ref{corspoinc},
\begin{equation}\label{eqsplogsobrad}
\beta_\mu^{SP}(t) \; \leq \; 2 \, e^{2C_{LS}(\mu_r)/t} \; + \; 2 \, \left(\frac{2 \, \mu(\rho^2)}{n+3}\right)^{(n-1)/2} \; e^{(n+1) \,  C_{LS}(\mu_r)/t} \, .
\end{equation}
Once again since we only have to look at $t<C_P(\mu)$, and using \eqref{eqpoincrad}, we obtain the following worse bound
\begin{equation}\label{eqsplogsobradbis}
\beta_\mu^{SP}(t) \; \leq \; 2 \, \left(e^{-(n-1) C_{LS}(\mu_r)/\max(C_P(\mu_r),\frac{\mu(\rho^2)}{n-1} )\, } \; + \;  \left(\frac{2 \mu(\rho^2)}{n+3}\right)^{(n-1)/2}\right) \; e^{(n+1) \,  C_{LS}(\mu_r)/t} \, .
\end{equation}
But we can use the following homogeneity property of the log-Sobolev and the Poincar\'e inequalities: defining for $\lambda >0$, $$\int \, f(z) \, \mu_\lambda(dz) \, = \int \, f(\lambda z) \, \mu(dz) \, ,$$ it holds $C_{LS}(\mu_\lambda)=\lambda^2 \, C_{LS}(\mu)$ (the same for $C_P(\mu)$). 

\noindent Looking at the pre-factor in \eqref{eqsplogsobradbis}, we see that making $\lambda$ go to $0$, the second term goes to $0$ while the first one is unchanged. Using the homogeneity properties for both $\mu$ and $\mu_r$, and using lemma \ref{lemspls} again, we have thus obtained
\begin{equation}\label{eqsplogsobrad1}
C_{LS}(\mu) \, \leq \, 64 \, \left((n+1)  \, C_{LS}(\mu_r) \, + \, C \, \max\left(C_P(\mu_r),\frac{\mu(\rho^2)}{n-1}\right)\right) \, ,
\end{equation}
where $$C=\ln\left(1\vee \, \frac{1+2e^2}{2} \, e^{-(n-1) \, C_{LS}(\mu_r)/\max(C_P(\mu_r),\frac{\mu(\rho^2)}{n-1} \, )}\right) \, .$$ Of course in many (almost all) situations $C=0$. Using $C_P(\mu_r) \leq C_{LS}(\mu_r)/2$ it is not very difficult to show that $C \neq 0$ if and only if $$\mu(\rho^2) \geq \frac{2(n-1)^2 \, C_{LS}(\mu_r)}{1+e^2} \, ,$$ in which case $C \leq \ln(1+e^2)$.

\noindent These results extend to the ``almost'' radial situation, using the standard perturbation result for a (super)-Poincar\'e inequality or a log-Sobolev inequality, as explained in the next Corollary
\begin{corollary}\label{coralmostrad}
Assume that $\mu(dx) = \mu_r(d\rho) \, \mu_a(d\theta)$ with $$m\leq \left | \left | \frac{d\mu_a}{d\sigma_n}\right |\right |_\infty \leq M$$ where $\sigma_n$ is the uniform probability measure on $\mathbb S^{n-1}$. Then
$$C_P(\mu) \, \leq \, C_P(\mu_r) \, + \, \frac Mm \, \frac{\mu_r(\rho^2)}{n-1} \, .$$ If $n\geq 4$ and $\mu_r$ satisfies a log-Sobolev inequality then so does $\mu$ and $$C_{LS}(\mu) \leq \, 64 \; \frac{M}{m} \, (n+1) \, C_{LS}(\mu_r)  \, ,$$ except if $\mu(\rho^2) \geq \frac{2(n-1)^2 \, C_{LS}(\mu_r)}{1+e^2}$ in which case $$C_{LS}(\mu) \leq \, 64 \; \frac{M}{m} \, \left((n+1) \, C_{LS}(\mu_r)  \, + \, \ln(1+e^2) \, \frac{\mu(\rho^2)}{n-1}\right) \, .$$ The previous two bounds amounts to the existence of an universal constant $C$ such that $$C_{LS}(\mu) \leq \, C \; \frac{M}{m} \, \max \left(n \, C_{LS}(\mu_r) \; , \; \frac{\mu(\rho^2)}{n-1}\right) \, .$$
\end{corollary}
\medskip

We may of course adapt the above proof to characterize the logarithmic Sobolev inequality starting from Theorem \ref{thmspoinc}, i.e:
\begin{theorem}\label{thmlssup}
Assume that $n\geq 4$, $\mu(dx) = \mu_r(d\rho) \, \mu_a(d\theta)$ and that $\mu_a$ satisfies a super Poincar\'e inequality with $\beta^{SP}_{\mu_a}(t) = C_a \, t^{-\kappa}$. Then $$C_{LS}(\mu) \leq \, 64 \, \left(c(\kappa) \, C_{LS}(\mu_r) \, + \, c \,  \mu_r(\rho^2) \, C_P(\mu_a)\right) \, .$$
\end{theorem}
\medskip

\section{The Logarithmic-Sobolev inequality in the nearly radial case.}\label{secls}

\subsection{An alternate direct approach. \\ \\}\label{secaltls}
Instead of using super Poincar\'e inequalities, let us try to directly mimic Bobkov's tensorization  in the case of log-Sobolev. First
\begin{eqnarray}\label{eqlogsob1}
\int \, (f^2 \, \ln(f^2))(\rho \, \theta) \, \mu_r(d\rho) &\leq& C_{LS}(\mu_r) \, \int \, |\nabla_\theta f(\rho \theta)|^2 \, \mu_r(d\rho) + \nonumber \\ & & + \, \left(\int \, f^2(\rho \theta) \, \mu_r(d\rho)\right) \, \ln \left(\int \, f^2(\rho \theta) \, \mu_r(d\rho)\right) \, ,
\end{eqnarray}
so that integrating with respect to $\mu_a$ we get 
\begin{equation}\label{eqlogsobnew}
\Ent_\mu(f^2) \, \leq C_{LS}(\mu_r) \, \mu(|\nabla_\theta f|^2) \ + \Ent_{\mu_a}(w^2)
\end{equation}
with $$w(\theta)= \left(\int \, f^2(\rho \theta) \, \mu_r(d\rho)\right)^{\frac 12} \, ,$$ after having remarked that $$\mu(f^2) \, \ln(\mu(f^2)) = \mu_a(w^2) \, \ln(\mu_a(w^2)) \, .$$
We are thus facing a difficulty. Indeed  $$\nabla w(\theta)=\frac{\int \, f(\rho \theta) \, \rho \, \nabla_{\theta^\perp} f(\rho \theta) \, \mu_r(d\rho)}{\left(\int \, f^2(\rho \theta) \, \mu_r(d\rho)\right)^{\frac 12}} \, .$$ Hence if we use the classical log-Sobolev inequality
\begin{eqnarray}\label{eqlogsob2}
\Ent_{\mu_a}(w^2) &\leq& C_{LS}(\mu_a) \, \mu_a(|\nabla w|^2) \,  \nonumber \\ &\leq& \, C_{LS}(\mu_a) \, \int \, \frac{\left(\int \, f(\rho \theta) \, \rho \, \nabla_{\theta^\perp} f(\rho \theta) \, \mu_r(d\rho)\right)^2}{\int \, f^2(\rho \theta) \, \mu_r(d\rho)} \, \mu_a(d\theta) \, .
\end{eqnarray}
Using Cauchy-Schwarz inequality in two different ways we have obtained
\begin{proposition}\label{propls}
If $\mu(dx) = \mu_r(d\rho) \, \mu_a(d\theta)$, then 
\begin{enumerate}
\item $$\Ent_\mu(f^2) \leq \max \, \left(C_{LS}(\mu_r) \, , \,  \parallel \rho\parallel^2_{\mathbb L^\infty(\mu_r)} \, C_{LS}(\mu_a)\right) \, \mu(|\nabla f|^2) \, .$$
\item  $$\Ent_\mu(f^2) \leq \, \max \, \left(C_{LS}(\mu_r) \, , \, C_{LS}(\mu_a)\right) \, \mu((1 \vee \rho)^2 \, |\nabla f|^2) \, .$$
\end{enumerate}
\end{proposition}
Since the log-Sobolev inequality is preserved when translating $\mu$, the first inequality in Proposition \ref{propls} implies
\begin{corollary}\label{corls}
If $\mu(dx) = \mu_r(d\rho) \, \mu_a(d\theta)$ and $\mu$ is supported by a bounded set $K$, then $$C_{LS}(\mu) \leq \max \left(C_{LS}(\mu_r) \, , \, \frac{diam^2 K}{4} \, C_{LS}(\mu_a)\right) \, .$$
\end{corollary}
The second inequality in Proposition \ref{propls} is a \emph{weighted} log-Sobolev inequality, with weight $(1\vee \rho)^2$, which is much weaker than the log-Sobolev inequality. These inequalities have been studied for instance in \cite{Wang08,CGW2}. Consequences in terms of concentration, rate of convergence or transport are in particular discussed in section 3 of \cite{CGW2}. The weight $(1\vee \rho)^2$ is however too big for being really interesting. In particular this weighted log-Sobolev inequality does not imply a Poincar\'e inequality in whole generality.
\medskip

Now consider the almost radial situation. According to section \ref{secsphere}, $C_{LS}(\sigma_n)\leq \frac{2}{n-1}$ for $n\geq 2$. It thus follows from corollary \ref{corls} and perturbation arguments
\begin{corollary}\label{corlsrad}
For all $\mu(dx) = \mu_r(d\rho) \, \mu_a(d\theta)$ supported by some bounded set $K$ and satisfying $$m\leq \left | \left | \frac{d\mu_a}{d\sigma_n}\right |\right |_\infty \leq M \, ,$$ it holds
\begin{equation}\label{eqlogentropenergy6}
C_{LS}(\mu) \leq \, \max \, \left(C_{LS}(\mu_r) \; , \; \frac{M}{m} \; \frac{diam^2 K}{2(n-1)}\right) \, .
\end{equation} 
\end{corollary}
\medskip

\subsection{The (almost) radial case. \\ \\}\label{radlsvia1}
Assume for a moment that $\mu_a=\sigma_n$. Rewrite \eqref{eqlogsobnew},  
\begin{equation}\label{eqlogsobnew2}
\mu(f^2 \, \ln(f^2)) \, \leq \, C_{LS}(\mu_r) \, \mu(|\nabla_\theta f|^2) \, + \int \, g^q \, \ln^q(g) \, d\sigma_n
\end{equation}
with $g(\theta) = \left(\int \, f^2(\rho \, \theta) \, \mu_r(d\rho)\right)^{1/q}$ and $1\leq q \leq 2$.

Instead of the usual log-Sobolev inequality, we may now use the $\mathbb L^q$ log-Sobolev inequality for $\sigma_n$ we have obtained in section \ref{secsphere}. It thus holds
\begin{equation}\label{eqlsvia1}
\int \, g^q \, \ln(g^q) \, d\sigma_n \, \leq \, c\, (n-1)^{1-q} \, \int \, |\nabla_{\theta^\perp} g|^q \, d\sigma_n \, + \, \left(\int \, g^q \, d\sigma_n\right) \, \ln\left(\int \, g^q \, d\sigma_n\right) \, , 
\end{equation}
so that $$\Ent_\mu(f^2) \, \leq \,  C_{LS}(\mu_r) \, \mu(|\nabla_\theta f|^2) \, + \, c \, (n-1)^{1-q} \, \int \, |\nabla_{\theta^\perp} g|^q \, d\sigma_n \, .$$
Now
\begin{eqnarray}\label{eqlsvia111}
\int \, |\nabla_{\theta^\perp} g|^q \, d\sigma_n &=& \frac{2^q}{q} \, \int \, \left(\int \, f^2(\rho \, \theta) \, \mu_r(d\rho)\right)^{1-q} \, \left(\int \, |f| \, |\nabla_{\theta^\perp} f| \, \rho \, \mu_r(d\rho)\right)^q \, d\sigma_n \nonumber \\ &\leq& \frac{2^q}{q} \parallel \rho\parallel_\infty^q \,  \int \, \left(\int \, f^2 \, d\mu_r\right)^{1- \frac q2} \; \left(\int \, |\nabla_{\theta^\perp} f|^2 \, d\mu_r\right)^{\frac q2} \, d\sigma_n \, \nonumber \\ &\leq& \frac{2^q}{q} \parallel \rho\parallel_\infty^q \,  \left(\int \, f^2 \, d\mu\right)^{1- \frac q2} \; \left(\int \, |\nabla_{\theta^\perp} f|^2 \, d\mu\right)^{\frac q2} \, 
\end{eqnarray}
where we have used H\"{o}lder's inequality in the last line. If $q=1$ we may also use the bound
\begin{equation}\label{eqlsvia12}
\int \, |\nabla_{\theta^\perp} g| \, d\sigma_n \leq  2  \int \, \left(\int \, f^2 \, \rho^2 \, d\mu_r\right)^{1/2} \; \left(\int \, |\nabla_{\theta^\perp} f|^2 \, d\mu_r\right)^{\frac 12} \, d\sigma_n \, .
\end{equation}
\medskip

First use \eqref{eqlsvia111}. Recall Rothaus lemma $$\Ent_\mu(f^2)  \, \leq \, \Ent_\mu((f-\mu(f))^2) \, + \, 2 \, \Var_\mu(f) \, .$$ We may thus replace $f$ by $f-\mu(f)$ and use Poincar\'e's inequality in order to get
\begin{equation}\label{eqlsvia111b}
\int \, |\nabla_{\theta^\perp} g|^q \, d\sigma_n \leq \frac{2^q}{q} \parallel \rho\parallel_\infty^q \, C_P(\mu) \, \left(\int \, |\nabla_{\theta^\perp} f|^2 \, d\mu\right) \, .
\end{equation}
Gathering all these results, using again Poincar\'e inequality for bounding the variance and Theorem \ref{thmpoinc}, we thus have if $\mu$ is supported by $K$ 
\begin{eqnarray}\label{eqraslsq}
\Ent_\mu(f^2) &\leq& C_{LS}(\mu_r) \, \int \, |\nabla_\theta f(\rho \theta)|^2 \, \mu_r(d\rho) + \, 2 C_P(\mu) \,  \left(\int \, |\nabla f|^2 \, d\mu\right)
 \nonumber \\ && \quad + \, c(q) \, (n-1)^{1-q} \, diam^q(K) \, C_P(\mu) \,  \left(\int \, |\nabla_{\theta^\perp} f|^2 \, d\mu\right) \, . \end{eqnarray}
It is easy to check that the best value of $q$ is $2$ if $diam K \leq (n-1)$ and $1$ if $diam K \geq (n-1)$. We have thus shown
\begin{theorem}\label{thmlsborne}
There exists an universal $c$ such that, if $\mu$ is radial with bounded support $K$, $$C_{LS}(\mu) \leq \, 2 \, \max\left(C_P(\mu_r) \, , \, \frac{\mu(\rho^2)}{n-1}\right) \, + \, \max \left(C_{LS}(\mu_r) \, , \, c \, \min\left(diam K \, , \, \frac{diam^2(K)}{n-1}\right)\right) \, .$$ If $\mu$ is almost radial it is enough to multiply the previous bound by $M/m$.
\end{theorem}
\medskip

Now we come back to \eqref{eqlsvia12}.

Let us consider $\int \, f^2 \, \rho^2 \, d\mu$ integrating first w.r.t. $\mu_r$.  By using the variational description of the relative entropy we have for any $t>0$
\begin{equation}\label{eqentropls1}
\int \, f^2(\rho \, \theta) \, \rho^2 \, \mu_r(d\rho) \, \leq \, \Ent_{\mu_r}(f^2(. \, \theta)) \, + \, \frac 1t \, \ln \left(\int \, e^{t\rho^2} \, \mu_r(d\rho)\right) \, \left(\int \, f^2(\rho \, \theta) \, \mu_r(d\rho)\right) \, .
\end{equation} 
We may of course stop here to get a first control of the logarithmic Sobolev constant of $\mu$ by uing recentering and Rothaus lemma (see the end of the argument) but les us see how using the same approach will provide us with an easy to apprehend formulation of the logarithmic Sobolev constant. We first use $$\int \, f^2 \, \rho^2 \, d\mu_r \, \leq \, 2 \, \int \, f^2 \, (\rho-\mu(\rho))^2 \, d\mu \, + \, 2 \, \mu_r^2(\rho) \, \int \, f^2 \,  d\mu_r \, ,$$ and using again the variational description of the relative entropy we have for any $t>0$,
\begin{equation}\label{eqentropls}
\int \, f^2(\rho \, \theta) \, (\rho-\mu(\rho))^2 \, \mu_r(d\rho) \, \leq \, \Ent_{\mu_r}(f^2(. \, \theta)) \, + \, \frac 1t \, \ln \left(\int \, e^{t(\rho-\mu(\rho))^2} \, \mu_r(d\rho)\right) \, \left(\int \, f^2(\rho \, \theta) \, \mu_r(d\rho)\right) \, .
\end{equation}
But since $\rho \mapsto \rho - \mu_r(\rho)$ is $1$-Lipschitz and of $\mu_r$ mean equal to $0$,  it is known (see e.g. \cite{bob99} formula (4.9)) that for $t<1/C_{LS}(\mu_r)$ $$\int \, e^{t(\rho-\mu(\rho))^2} \, \mu(d\rho) \, \leq \, \frac{1}{\sqrt{1-t\, C_{LS}(\mu_r)}} \, .$$
For $t=1/2 \, C_{LS}(\mu_r)$ we thus deduce
\begin{equation*}\label{eqentropls2}
\int \, f^2(\rho \, \theta) \, (\rho-\mu(\rho))^2 \, \mu_r(d\rho) \, \leq \, C_{LS}(\mu_r) \, \int \, |\nabla_\theta f(\rho \, \theta)|^2 \, \mu_r(d\rho) \, + \, \ln(2) \, C_{LS}(\mu_r) \, \int \, f^2(\rho \, \theta) \, \mu_r(d\rho) \, .
\end{equation*}
Finally
\begin{equation}\label{eqentropls3}
\int \, f^2 \, \rho^2 \, d\mu \, \leq \, 2 \, C_{LS}(\mu_r) \, \int \, |\nabla_\theta f|^2 \, d\mu \, + \, 2(\ln(2) \, C_{LS}(\mu_r) \, + \, \mu_r^2(\rho)) \, \int \, f^2\, d\mu \, .
\end{equation}
Replacing $f$ by $f-\mu(f)$ and using the Poincar\'e inequality, we thus deduce $$\int \, (f-\mu(f))^2 \, \rho^2 \, d\mu \, \leq \, 2 \, C_{LS}(\mu_r) \, \int \, |\nabla f|^2 \, d\mu \, + \, 2(\ln(2) \, C_{LS}(\mu_r) \, + \, \mu_r^2(\rho))  \, C_P(\mu) \, \int \, |\nabla f|^2\, d\mu \, .$$ Gathering all what precedes we get $$\Ent_\mu((f-\mu(f))^2) \, \leq \, A \, \mu(|\nabla f|^2)$$ with 
\begin{equation}\label{eqyep}
A \, = \, C_{LS}(\mu_r)  + \, c \, \left(C_{LS}(\mu_r) \, + \, 2(\ln(2) \, C_{LS}(\mu_r) + \mu^2(\rho)) \, C_P(\mu))\right)^{\frac 12} \, \, .
\end{equation}
To conclude it remains to use Rothaus lemma again. We have thus obtained after some simple manipulations using in particular $2C_P\leq C_{LS}$, the concavity of the square root, and the homogeneity of the inequalities w.r.t. dilations as we did in order to get \eqref{eqsplogsobrad1} but this time with $\lambda \to +\infty$, and finally \eqref{eqpoincrad} (replacing for simplicity the $\max$ by the sum)
\begin{theorem}\label{thmyep}
There exists an universal constant $c$ such that, for all $\mu(dx) = \mu_r(d\rho) \, \mu_a(d\theta)$ satisfying $$m\leq \left | \left | \frac{d\mu_a}{d\sigma_n}\right |\right |_\infty \leq M \, ,$$ it holds $$C_{LS}(\mu)) \leq \, c \, \frac Mm \, \left(C_{LS}(\mu_r) + \mu(\rho) \max \left(C_P(\mu_r) \, , \, \frac{\mu(\rho^2)}{n-1}\right)^{1/2} \right) \, .$$
Alternatively we have 
$$C_{LS}(\mu)) \leq \, c \, \frac Mm \, \left(C_{LS}(\mu_r) + \left(\inf_{t>0}\frac1t\ln\left(\int e^{t\rho^2}\mu_r(d\rho)\right)\right)^{1/2} \max \left(C_P(\mu_r) \, , \, \frac{\mu(\rho^2)}{n-1}\right)^{1/2} \right) \, .$$
\end{theorem}
Notice that if $\mu$ is supported by a bounded set $K$, we recover only partially the conclusion of Corollary \ref{corlsrad}.

\begin{remark}
Of course the previous results are much better, in terms of the dimension dependence, than Corollary \ref{coralmostrad} since the pre-factor of $C_{LS}(\mu_r)$ is ``dimension free'' (more precisely can be bounded from above by an universal constant), while the dimension appears in front of $C_{LS}(\mu_r)$ in 
Corollary \ref{coralmostrad}. Let us remark also that the constant appearing in the second formulation is close from the one obtained by Bobkov in dimension one \cite{bob99}. It will appear again in the next Section.  Remark also that contrary to the corollary, the proof cannot be extended to more general cases, except if the angular part $\mu_a$ satisfies a similar $\mathbb L^1$ log-Sobolev inequality. \hfill $\diamondsuit$
\end{remark}
\smallskip

To finish this section, let us remark that one could also get another way to control \eqref{eqlsvia12} by using Lyapunov conditions rather than using the Logarithmic Sobolev inequality for the radial part. Of course, by \cite{CGhitlyap}, one also knows that in our setting a logarithmic Sobolev inequality is equivalent to some Lyapunov type conditions. However we will see that in order to control \eqref{eqlsvia12} one needs a slightly weaker inequality. Indeed let us suppose here that there exists $W\ge 1$, $a,b>0$ such that
\begin{equation}
\label{lyapcond}
\rho^2\le -a\frac{L_\rho W}{W}+b
\end{equation}
where $L_\rho f=f''-(V'_r+\frac{n-1}{\rho})f'$ is the generator corresponding to the radial part of $\mu$. Recall now that we need to control
$$\int f^2\rho^2\mu_r(d\rho).$$
Using \eqref{lyapcond} we get
$$\int f^2\rho^2\mu_r(d\rho)\le a\int f^2\frac{-LW}W\mu_r(d\rho)+b\int f^2\mu_r(d\rho).$$
The first term is easily dealt with, using integration by parts or a large deviations argument as in \cite{CGWW}:
$$\int f^2\frac{-LW}W\mu_r(d\rho)\le \int|\nabla_\theta f(\rho\theta)|^2\mu_r(d\rho).$$
We then use the same trick as before, i.e centering and Rothaus lemma to get 
\begin{proposition}
Assume \eqref{lyapcond}. There exists an universal constant $c$ such that, for all $\mu(dx) = \mu_r(d\rho) \, \mu_a(d\theta)$ satisfying $$m\leq \left | \left | \frac{d\mu_a}{d\sigma_n}\right |\right |_\infty \leq M \, ,$$ it holds $$C_{LS}(\mu)) \leq \, c \, \frac Mm \, \left(C_{LS}(\mu_r) +\sqrt{a}+\sqrt{b} \max \left(C_P(\mu_r) \, , \, \frac{\mu(\rho^2)}{n-1}\right)^{1/2} \right) \, .$$
\end{proposition}
\begin{remark}
Let us make a few comments about the Lyapunov condition \eqref{lyapcond}. It has been shown in \cite{CGW1} that it is a sufficient condition for Talagrand inequality, and that there exists examples satisfying this condition and not a logarithmic Sobolev inequality. Nevertheless, we need for the first part of the proof that the radial part satisfies a logarithmic Sobolev inequality. More crucial are the values of the constants $a$ and $b$ with respect to the dimension. If $a$ can be chosen dimension free in usual cases, say $V_r(\rho)=\rho^k$, $b$ is then of order $n$ and we then get an additional $\sqrt{n}$ factor for the logarithmic Sobolev constant in this case. Of course, it is surely better than the $n$ factor by using Super-Poincar\'e inequality. \hfill $\diamondsuit$
\end{remark}

\section{Some applications in the (almost) radial case.}\label{secapp}

The main interest of the previous results is that they reduce the study of functional inequalities for $\mu$ to the one of its radial part $\mu_r$ which is supported by the half line. For such one dimensional measures explicit criteria of Muckenhoupt (or Hardy) type are well known \cite{Ane,BaGLbook,BRstud,BCR3,BCR1}). Let us recall the case of the Poincar\'e inequality (see e.g. \cite{BaGLbook} Theorem 4.5.1) and of the log-Sobolev inequality (see \cite{BCR3} Theorem 7 with $T(u)=2u$, or \cite{CGG} Proposition 2.4 for a slightly different version).
\begin{proposition}\label{prophardy}
Assume that $\mu_r$ is absolutely continuous w.r.t. Lebesgue measure with density $\rho_r$. Assume in addition that the support of $\mu_r$ is an interval $I$. Let $m$ be a median of $\mu_r$, then $$\frac{1}{2} \, \max(b_-,b_+) \, \leq \, C_{P}(\mu_r) \, \leq \, 4 \, \max(b_-,b_+)$$ and $$\frac{1}{12} \, \max(B_-,B_+) \, \leq \, C_{LS}(\mu_r) \, \leq \, 10 \, \max(B_-,B_+)$$ where $$b_- = \sup_{x\in I ,  0\leq x< m} \, \mu_r([0,x]) \,  \int_x^m \, \frac{1}{\rho_r(\rho)} \, d\rho$$ $$b_+ = \sup_{x\in I , x> m} \mu_r([x,+\infty)) \, \int_m^{x} \, \frac{1}{\rho_r(\rho)} \, d\rho$$
while
$$B_- = \sup_{x\in I , 0\leq x< m} \, \mu_r([0,x]) \, \ln\left(1 + \frac{1}{\mu_r([0,x])}\right) \, \int_x^m \, \frac{1}{\rho_r(\rho)} \, d\rho$$
 $$B_+ = \sup_{x\in I , x> m} \, \mu_r([x,+\infty)) \, \ln\left(1 + \frac{1}{\mu_r([x,+\infty))}\right) \, \int_m^{x} \, \frac{1}{\rho_r(\rho)} \, d\rho \, .$$ 
 \end{proposition}
Several methods are known to furnish estimates for quantities like $\mu_r([a,+\infty[)$ or $\int (1/\rho_r)$ (see e.g. \cite{Ane} chapter 6.4). 
\begin{remark}\label{reml11d}
We will not study in more details the $\mathbb L^1$ inequalities on the real (half) line. However, because it is immediate, let us only give an upper bound for $C_C$. 
\begin{proposition}\label{propCC1}
Let $\nu(dx)=e^{-W(x)} dx$ be a probability measure on $\mathbb R$. 

Then $C_C(\nu) \leq \max(b^1_-,b^1_+)$ where $$b^1_- = \sup_{t \in \mathbb R} \, e^{W(t)} \, \nu(]-\infty,t]) \quad \textrm{ and } \quad b^1_+= \sup_{t \in \mathbb R} \, e^{W(t)} \, \nu([t,+\infty[) \, .$$
\end{proposition}
\begin{proof}
We simply write for any $a$, 
\begin{eqnarray*}
\nu(|f - m_\nu(f)|) &\leq& \nu(|f-f(a)|) = \int \, \left|\int_a^x \, f'(t) \, dt\right| \, \nu(dx) \\ &\leq& \int_{-\infty}^a \, \int_x^a \, |f'(t)| dt \, \nu(dx) + \int_a^{+\infty} \, \int_a^x \, |f'(t)| dt \, \nu(dx)\\ &\leq& \int_{-\infty}^a \, |f'(t)| \, \nu(]-\infty,t]) \, dt \, + \, \int_a^{+\infty} \, |f'(t)| \, \nu([t,+\infty[) \, dt \, ,
\end{eqnarray*}
and then write $dt = e^{W(t)} \, \nu(dt)$ to get the result.
\end{proof}
It is possible to get a lower bound, and bounds for $C_{LS1}$ using ad-hoc Orlicz spaces as in \cite{bobhoud} and the previous trick in one dimension.  \hfill $\diamondsuit$
\end{remark}

\begin{remark}\label{remlower}
Before to study some simple examples, let us say a word about ``optimality'' in the radial case.

Notice first that if $f(x)=g(|x|)$, $\nabla f(x)= g'(|x|) \, \frac{x}{|x|}$ so that $|\nabla f|^2(x)=(g'(|x|))^2$. It follows that $C_P(\mu) \geq C_P(\mu_r)$ and similarly for the log-Sobolev constant.

If we take $f(x)= \sum_{i=1}^n x_i/|x|$, $\Var_\mu(f)=1$ while $\mu(|\nabla f|^2)= (n-1) \, \mu(1/\rho^2)$ so that $C_P(\mu) \geq 1/((n-1) \, \mu(1/\rho^2)$. The latter is smaller than $\mu(\rho^2)/(n-1)$ but both quantities are comparable in many situations. We shall see it on the examples below.
\hfill $\diamondsuit$
\end{remark}
\medskip

In all the examples below, we will pay a particular attention to \textit{dimension dependence}, so that any sentence like ``the good order'' or ``good dependence'' has to be understood ``with respect to the dimension $n$.
\medskip

\begin{example}\label{exboule}\textbf{The uniform measure on an euclidean ball.}
\smallskip

Consider the simplest example for radial $\mu$ i.e. the uniform measure on the euclidean ball of radius $R$. The good order for the log-Sobolev constant has been derived in \cite{BLgafa} Proposition 5.3 by pushing forward the gaussian distribution onto the uniform one on the ball (also see Proposition 5.4 therein for the LSq inequality). 
\medskip

We shall obtain here similar bounds by using our results. First $$\mu_r(d\rho) = \frac{n \, \rho^{n-1}}{R^n} \, \mathbf 1_{0\leq \rho \leq R} \, d\rho$$ so that the mean, the (unique) median, the second moment, the Variance of $\mu_r$ are respectively:
$$\mu_n= \frac{n}{n+1} R, \; m_n=(1/2)^{1/n} \, R, \; \mu_r(\rho^2)=\frac{n}{n+2} \, R^2, \; v_n=\frac{n}{(n+2)(n+1)^2} \, R^2 \, .$$ It is then easily seen that $b_-$ and $b_+$ are both less than $\frac{R^2}{n(n-2)}$ provided $n>2$. The case $n=2$ can be handled separately. It follows that $C_P(\mu_r)\leq \, 4 \,  (R^2/n^2)$. We also have $C_P(\mu_r) \geq v_n$ so that $R^2/n^2$ is the good order. Finally $$C_P(\mu) \leq \max\left(\frac{4}{n^2} \, , \,  \frac{n}{(n+2)(n-1)}\right) \, R^2 \, .$$ This bound is not sharp, but asymptotically sharp. Indeed $\mu(1/\rho^2)= \frac{n}{n-2} \, R^{-2}$, so that according to the discussion in remark \label{remlower}, $C_P(\mu) \geq \frac{n-2}{n(n-1)} \, R^2$.

Notice that the upper bound better is exactly the one obtained in \cite{BJM} Theorem 1.2, while the lower bound is better than the one in \cite{BJM}. Theorem 1.2 in \cite{BJM} deals with general radial log-concave measures. by using dedicated tools for log-concave one dimensional measures, we shall come back to this later.
\medskip

We turn now to the log-Sobolev constant assuming first that $n \geq 3$.\\ \noindent First $$B_-= \sup_{0\leq x<m_n} \, \frac{x^2}{n(n-2)} \, \left(1 - (x/m_n)^{n-2}\right) \, \ln (1+(R/x)^n) \, ,$$ so that using $\ln(1+u^n) \leq \ln 2 + n \ln(u)$ for $u\geq 1$, we deduce $$B_- \leq \frac{R^2}{n-2} \; \sup_{u\geq 1}\left(\frac{1}{u^2} \, \left(\frac{\ln 2}{n} + \ln u\right)\right) \, \leq \, c_1 \, \frac{R^2}{n-2} \, ,$$ with $c_1=\frac{1}{2e} + \frac{\ln 2}{n} \leq 1$. Similarly, using $\ln(1+u^n)\geq n \ln u$ for $u\geq 1$ we have $$B_- \, \geq \, \frac{R^2}{n-2} \, \sup_{2^{1/n}\leq u} \, \left(\frac{\ln u}{u^2} \, \left(1- 2^{(n-2)/n} u^{-(n-2)}\right)\right) \, \geq \, c_2 \, \frac{R^2}{n-2} \, ,$$ with for instance $c_2= \ln 2/16$ obtained by choosing $u=4$. 

\noindent Next $$B_+ = \sup_{m_n<x\leq R} \, (1-(x/R)^n) \, \frac{R^2}{n(n-2)} \, \left(2^{(n-2)/n} - (R/x)^{n-2}\right) \, \ln\left(1 + \frac{1}{1-(x/R)^n}\right) \, ,$$ so that $$B_+ \leq \, \frac{R^2}{n(n-2)} \, \sup_{0<v\leq 1}\left(v \, \ln(1+(1/v))\right) \, \leq  \, \frac{R^2}{n(n-2)} \, .$$ Gathering all this we obtain that $$ \frac{\ln 2}{192} \, \frac{R^2}{n-2} \, \leq \, C_{LS}(\mu_r) \leq 10 \, \frac{R^2}{n-2} \, .$$ Note that choosing $f(x)=x$ one gets  the better lower bound $$C_{LS}(\mu_r) \geq \, R^2 \, \frac{n^2}{n+2} \, \left(\ln(1+\frac 2n) - \frac{2}{n+2}\right) \geq \, 2 \, R^2 \, \frac{n-2}{(n+2)^2} \, .$$
Using Corollary \ref{corlsrad} we have thus shown, since $C_{LS}(\mu) \geq C_{LS}(\mu_r)$, and after simple manipulations
\begin{proposition}\label{logsobboule}
For all $n\geq 3$ the uniform measure $\mu$ on the euclidean ball of radius $R$ satisfies $\frac{1}{20} \, \frac{2 \, (n-2) \, R^2}{(n+2)^2} \, \leq \, C_{LS}(\mu) \leq \, 10 \, \frac{R^2}{n-2}$.
\end{proposition}
\end{example}
\medskip

\begin{example}\label{exlogconcave}\textbf{Spherically symmetric log-concave measures.}
\smallskip

The previous example is a particular example of a radial log-concave measure, i.e. $\mu(dx) = e^{-V(|x|)} \, dx$ where $V$ is convex and non decreasing. Actually for what follows (in the radial situation) we do not need $V$ to be non decreasing, but still convex. 

For such measures the Poincar\'e constant was first studied by Bobkov in \cite{bobsphere}. Bobkov's result was improved by Bonnefont-Joulin-Ma in \cite{BJM} Theorem 1.2 who states that for such measures 
\begin{equation}\label{eqBJM}
\frac{\mu(\rho^2)}{n} \, \leq \, C_P(\mu) \, \leq \, \frac{\mu(\rho^2)}{n-1} \, .
\end{equation}
The case of the log-Sobolev constant was not really addressed in the specific radial situation, but rather for general log concave distributions. Define $$Orl(\mu) \, = \, \inf \{ \; t>0 \; ; \; \int \, \exp(|x - \mu(x)|^2/t^2) \, \mu(dx) \, \leq \, 2 \; \} \, .$$ Then, according to Bobkov's theorem 1.3 in \cite{bob99}, if $\mu$ is log-concave
\begin{equation}
C_{LS}(\mu) \leq C \, Orl^2(\mu)
\end{equation}
the right hand side being finite or infinite. When $\mu$ is supported by a bounded set $K$, this yields the rough bound $C_{LS}(\mu) \leq C \, diam^2(K)$. The latter has been improved in \cite{leevemplogsob} Theorem 8, where it is shown that $C_{LS}(\mu) \leq C \, diam(K)$ provided $\mu$ is isotropic (i.e. its covariance matrix equals identity). Notice that Bobkov's Corollary 2.3 in \cite{bob99} tells that
\begin{equation}\label{eqbob99}
C_{LS}(\mu) \, \leq \, 2 \, (C_P(\mu) + diam(K) \, \sqrt{C_P(\mu)}) \, ,
\end{equation}
so that, it also furnishes the diameter bound if the K-L-S conjecture is true.
\smallskip

Actually the specific radial case was only addressed in \cite{nolwen} where the author studies the isoperimetric profile of radial log concave distributions (see theorem 4 and theorem 5 therein). Connections between the log-Sobolev constant and the isoperimetric profile are strong in the log-concave situation (see \cite{emillogsob} Theorem 1.2) but the results of \cite{nolwen} are not easy to handle with. We will thus use our previous results to derive explicit bounds.
\medskip

Writing $\mu_r(d\rho)= n \, \omega_n \, \rho^{n-1} \, e^{-V(\rho)} \, d\rho$ with $\omega_n$ the volume of the unit euclidean ball, we see that $\mu_r$ is also log-concave. One can thus apply Bobkov's results in \cite{bob99}, starting with Proposition 4.4 therein
\begin{equation}\label{eqls1logconc}
\frac 34 \, Orl^2(\mu_r) \, \leq \, C_{LS}(\mu_r) \, \leq \, 48 \, Orl^2(\mu_r)
\end{equation}
where $$Orl(\mu_r) \, = \, \inf \{ \; t>0 \; ; \; \int \, \exp((\rho - \mu_r(\rho))^2/t^2) \, \mu_r(d\rho) \, \leq \, 2 \; \} \, .$$ 
As a byproduct we get thanks to Theorem \ref{thmyep} and \eqref{eqBJM}
\begin{proposition}\label{lslogconcrad}
If $\mu$ is a radial log-concave distribution, $$C_{LS}(\mu) \leq C \, \left(Orl^2(\mu_r) + \frac{\mu(\rho^2)}{\sqrt{n-1}}\right) \, .$$ Using \eqref{eqls1logconc} and remark \ref{remlower}, one can also derive some lower bound.
\end{proposition} 

If we assume in addition that $\mu$ is supported by a bounded set $K$, we may obtain other bounds using our Corollary \ref{corlsrad} (the same with Theorem \ref{thmlsborne}). First we may translate $\mu$ so that its support is included in the centered euclidean ball with radius $diam(K)/2$. But we may also remark that $$\mu_r(d\rho)= Z_n^{-1} \, e^{(n-1)\ln \rho - V(\rho)} \, d\rho= Z_n^{-1} \, e^{-W(\rho)} \, d\rho$$ with $$W''(\rho)= V''(\rho)+ \frac{n-1}{\rho^2} \geq \frac{4(n-1)}{diam^2(K)} \, .$$ According to Bakry-Emery criterion (see e.g. \cite{BaGLbook} Proposition 5.7.1) we deduce
\begin{equation}\label{eqbakem}
C_{LS}(\mu_r) \, \leq \, \frac{diam^2(K)}{2(n-1)} \, .
\end{equation}
Applying Corollary \ref{corlsrad} we have thus obtained
\begin{theorem}\label{thmlogconc}
For any radial log-concave probability measure $\mu$ whose support $K$ is bounded, $$C_{LS}(\mu) \, \leq \, C \, \frac{diam^2(K)}{(n-1)} \, .$$ This bound is sharp in the sense that a similar lower bound is true for the uniform measure on a ball (recall Proposition \ref{logsobboule}). \end{theorem}
When the diameter is less than $n$ this result is better than \cite{leevemplogsob}. Actually we do not need $\mu$ to be isotropic contrary to \cite{leevemplogsob}. In general there is no control on the diameter of $K$ for an isotropic log-concave measure (or more generally for an almost isotropic measure i.e. such that $Cov \leq Id$ in the sense of quadratic forms), except of course that $diam K \geq \sqrt n$. 
\end{example}
\bigskip

{\bf Acknowledgements.} This work has been (partially) supported by the Project EFI ANR-17-CE40-0030 of the French National Research Agency

\bibliographystyle{plain}
\bibliography{cglsrad1}
\end{document}